\documentclass[a4paper]{amsart}
\usepackage{amssymb, amsmath, amsthm}
\usepackage{stmaryrd, esint}
\usepackage[usenames, dvipsnames]{color}
\newtheorem{theorem}{Theorem}[section]
\newtheorem{lemma}[theorem]{Lemma}
\newtheorem{proposition}[theorem]{Proposition}
\newtheorem{corollary}[theorem]{Corollary}

\theoremstyle{definition}
\newtheorem{definition}{Definition}[section]

\theoremstyle{remark}
\newtheorem{remark}[definition]{Remark}

\numberwithin{equation}{section}

\newcommand{\abs}[1]{\lvert#1\rvert}

\newcommand{\norm}[1]{\left\lVert#1\right\rVert}

\newcommand{\R}{\mathbb{R}}

\newcommand{\N}{\mathbb{N}}
\newcommand{\Z}{\mathbb{Z}}

\newcommand{\bC}{\mathbf{C}}
\newcommand{\bP}{\mathbf{P}}
\newcommand{\bp}{\mathbf{p}}
\newcommand{\bM}{\mathbf{M}}
\newcommand{\bd}{\mathbf{d}}

\newcommand{\cone}{\times\mspace{-15mu}\times}

\DeclareMathOperator{\dist}{dist}
\DeclareMathOperator{\spt}{spt}

\def\a#1{\left\llbracket{#1}\right\rrbracket}

\newcommand\res{\mathop{\hbox{\vrule height 7pt width .3pt depth 0pt\vrule height .3pt width 5pt depth 0pt}}\nolimits}
 
\title[Uniqueness of tangent cones]{Uniqueness of tangent cones to boundary points of two-dimensional almost-minimizing currents}

\author[J. Hirsch]{Jonas Hirsch}
\address[Jonas Hirsch]{Universität Leipzig, Mathematisches Institut, Augustusplatz 10, 04109 Leipzig, GERMANY}
\email{jonas.hirsch@math.uni-leipzig.de}

\author[M. Marini]{Michele Marini}
\address[Michele Marini]{Scuola Internazionale Superiore di Studi Avanzati, via Bonomea, 265, 34136 Trieste, ITALY}
\email{michele.marini@sissa.it}

\begin{document}
\maketitle

\begin{abstract}
We prove that tangent cones at singular boundary points of a two-dimensional current almost area minimizing are unique. Following the ideas exposed by White in \cite{White}, the result is achieved by combining a suitable epiperimetric inequality and an almost-monotonicity formula for the mass at boundary points.
\end{abstract}

\section{Introduction}

Let $T$ be an $m$-dimensional area minimizing integer rectifiable current in $\R^{m+n}$ and let $(\iota_{x,r})_\sharp T$ denote the push-forward of $T$ via the map $\iota_{x,r}:\R^{m+n}\to\R^{m+n}$ given by $z\mapsto \frac{z-x}{r}$. By using monotonicity formula one can show that the blow-ups $(\iota_{x,r})_\sharp T$ converge, up to subsequences, as $r\to 0$ to some cone $T_x$, {\it i.e.} an integral area-minimizing current such that $(\iota_{0,r})_\sharp T_x=T_x$, for every $r>0$ (see, for instance, \cite{Simon}).
\par
It is natural to wonder whether the tangent cone $T_x$ is uniquely determined, or it depends on the choice of the subsequence $r_k\to 0$. This question about the structure of the singularities of minimal surfaces has turned out to be particularly challenging and an answer has been given only in some particular situations. When $m=1$, uniqueness of tangent cones has been established in \cite{AllardAlmgren}.
For higher dimensions a possible approach is to use an epiperimetric inequality, see for instance \cite{White, DeLellis, Spolaor, Reifenberg}
\par

The case when $m=2$ has been covered by B. White in the seminal paper \cite{White}. The author provides a proof of the uniqueness of tangent cones for two-dimensional area-minimizing currents. In particular it is shown that every tangent cone satisfies an {\it epiperimetric inequality} (see \cite[Definition 2]{White} and compare it with our Lemma \ref{lem.boundary epiperimetric I} in Section \ref{sec.epiperimetric}). A tangent cone $T_x$ satisfies an epiperimetric inequality provided the difference of the mass in a ball $B$ between $T_x$ and that of a minimal surface $H$ is bounded by the difference of the masses of $T$ and the cone generated by the boundary of $H$, namely
\[
\norm{H}(B)-\norm{T_x}(B)\le(1-\varepsilon)\left(\norm{0 \cone \partial H}(B) - \norm{T_x}(B)\right),
\]
whenever $\partial H\in \mathbf I_1(\partial B)$ is sufficiently close to $\partial (T\res B)$.
\par
A crucial step for the proof of uniqueness is to show a decay of the flat distance between $T$ and an area minimizing cone, by combining the epiperimetric inequality and the monotonicity formula.\\
\par

In \cite{DeLellis}, White's technique has been suitably adapted for the case {\it almost (area) minimizing} two dimensional currents, and the same result of uniqueness of tangent cones at singular point has been established in this more general setting.\\
\par

In this paper we consider boundary points of two-dimensional almost-minimizing currents and, by relying on the ideas and computations exposed in \cite{White} and \cite{DeLellis}, we prove that, again, tangent cones are uniques.
\par
Before stating our main theorem let us give the exact definition of almost-minimizers we will make use of.

\begin{definition}\label{def.almost minimzing}
An $m$-dimensional integer rectifiable current $T$ in $\R^{m+n}$ with boundary $\Gamma$ will be called \emph{almost (area) minimizing at $x \in \spt(T)$} if there are constants $C_0, r_0, \alpha_0>0$ such that 
\begin{equation}\label{eq:almost minimizing}
\norm{T}(B_r(x)) \le (1+C_0 r^{\alpha_0})\norm{T + \partial Q}(B_r(x))
\end{equation}
for all $0<r<r_0$ and all integral $(m+1)$-dimensional currents $Q$ supported in $B_r(x)$.
\\
The current is called \emph{almost (area) minimizing} in a open set $U$ if the current $T$ is \emph{almost (area) minimizing at each $x \in \spt(T)\cap U$}.
\end{definition}

For such currents, in Section \ref{sec.final}, we will show the following
\begin{theorem}\label{thm.uniqueness}
Assume that $T$ is a two-dimensional integer rectifiable current in $\R^{2+n}$ with $C^{1,\alpha}$-boundary $\Gamma$. If $T$ is almost (area) minimizing in some open set $U$, then, for every $x\in U\cap \Gamma$, there exists a unique area minimizing cone $S$ such that $0\in\partial S$ and
\[
(\iota_{x,r})_\sharp T\to S,
\]
in the sense of currents.
\end{theorem}

To prove Theorem \ref{thm.uniqueness} we closely follow the strategy explained in \cite{White} and \cite{DeLellis}. In particular, in Section \ref{sec.monotonicity}, we derive an {\it almost monotonicity formula} (see Proposition \ref{prop:almost monotonicity formula}) for boundary points which will allow us to deal with blow-ups and to prove the decay of some flat norm along the blow up sequence.
\par
In Section \ref{sec.epiperimetric}, we establish an epiperimetric inequality for almost minimizing currents by adapting White's construction of a comparison surface. 
\par
In Section \ref{sec.final}, we finally show how to combine the results obtained in the first part of the paper in order to prove a refined version of Theorem \ref{thm.uniqueness}.
\par
In Section \ref{sec.approximation}, (see Lemma \ref{lem:Lipschitz approximation}) we explain how to slightly modify a rectifiable curve to obtain new ones with bounded Lipschitz constants, taking the same boundary values, and preserving symmetries.
\par
We conclude the paper by showing three instances of currents satisfying the assumptions of Definition \ref{def.almost minimzing}, compare also with the examples provided in \cite{DeLellis}.\\

We refer to the booktexts \cite{Federer} and \cite{Simon} for the notations and the basic definitions. In particular, we will use the short hand notation $T_{x,r}$ for $(\iota_{x,r})_\sharp T$ and only $T_r$ if $x=0$. The euclidean distance from a point $y$ will be denoted by $d_y$ i.e. $d_y(x):= \abs{x-y}$. In case of $y=0$ we will only write $d$ i.e. $d(x)=d_0(x)$. We will use a couple of times the projection $i$ onto the sphere i.e. $i(x) = \frac{x}{\abs{x}}$. Finally, we recall the definition of the flat distance between $T,S \in \mathbf{I}_m(B_{R+1})$, compare \cite[Section 6.7]{Simon}, 
\begin{equation}\label{eq: flat norm distance}
\bd_{B_R}(T,S)= \inf\{ \norm{R}(B_R) + \norm{Q}(B_R) \colon T-S = R+ \partial Q \text{ in } B_{R+1}\},
\end{equation}
with $R \in \mathbf{I}_m(B_{R+1})$ and $Q \in \mathbf{I}_{m+1}(B_{R+1})$.

\section*{Acknowledgements}
The authors would like to thank SISSA, for the support to the first author to visit Trieste. Furthermore they thank the HIM institute for their kind hospitality during the trimester program "evolution of interface". During these visits parts of the project had been discussed. The work of the second author is supported by the MIUR-grant {\it``Geometric
Variational Problems''} (RBSI14RVEZ).

\section{Almost Monotonicity at boundary points}\label{sec.monotonicity}

Since the results presented in this section are not affected by the dimension of the current, we will present them in full generality. Throughout this section, then, $T$ will be an $m$-dimensional almost-minimizing current in $\R^{m+n}$.\\

The following observation concerning the existence of a diffeomorphism that straightens the boundary will be a helpful tool throughout all the arguments. 
%Before proving the monotonicity formula for almost minimizing currents at boundary points let us explain how it is possible to "straighten" the boundary simply using the following observation. This will simplify a lot of computations. 

\begin{lemma}\label{lem.straight boundary}
Let $\Gamma$ be a $C^{1,\alpha_1}$ submanifold of dimension $m-1$. We assume that $\Gamma \res B_{r_1}$ is the graph of an entire function $\gamma \in C^{1,\alpha_1}(\R^{m-1}, \R^{n+1})$ with $\gamma(0)=0, D\gamma(0)=0$ i.e. $\Gamma \res B_{r_1} = \mathbf{G}_\gamma \res B_{r_1}$. If	$\norm{\gamma}_{C^{1,\alpha_1}}<\epsilon_1$. Then there exists a function $\phi: B_{r_1}(x) \to B_{r_1}(x)$ and constants $C_1, \alpha_1$ depending on $n, m, \Gamma$ with the properties that
\begin{itemize}
	\item[(i)] $\phi(x)=x$ and $\phi(\partial B_r) = \partial B_r$ for all $0<r< r_1$;
	\item[(ii)] $\phi(\Gamma\cap B_{r_1}) = \Gamma_0\cap B_{r_1}$ where $\Gamma_0=\{ x_l =0 \colon l \ge m \}$;
	\item[(iii)] $\frac{\abs{\phi(x)}}{\abs{x}} + \norm{D\phi(x)-\mathbf{1}}\le C_1 \abs{x}^{\alpha_1}$ for all $\abs{x}<r_1$.   
\end{itemize}
\end{lemma}

\begin{proof}
We will use the notation $x=(y,z) \in \R^{m-1}\times \R^{n+1}$, furthermore we fix a small angle $0<\theta< \frac\pi8$. 	
Let $\Phi_1$ be smooth approximation of 
\[\tilde{\Phi}_1(x):= \begin{cases}
 x & \text{ if } \abs{z} \ge \sin(2\theta) \abs{x}\\
 (\cos(2\theta) \frac{y}{\abs{y}}, z) 	& \text{ if } \abs{z} \le \sin(2\theta) \abs{x}
 \end{cases}
 \]
 such that $\Phi_1(x) = \tilde{\Phi}_1(x)$ if $\abs{x}>\frac{1}{10}$ and $\abs{z}\le \sin(\theta) \abs{x}$ or $\abs{z} \ge \sin(3\theta) \abs{x}$. Now let $\Phi(x):=\abs{x} \Phi_1(\frac{x}{\abs{x}})$ its one-homogenous extension. Note that $\Phi$ is smooth outside of $0$ and Lipschitz continuous on $\R^{n+m}$ and 
 \[ \Phi: \partial B_r \cap \{ \abs{z}\le \sin(\theta) r \} \to \partial \mathbf{C}_{\sin(\theta)r} \cap \{ \abs{z}\le \sin(\theta) r \}\,\]
 is a smooth diffeomorphism.
 Here $\mathbf{C}_{\sin(\theta)r}$ denotes the cylinder $\{ \abs{y}= \sin(\theta) r \}$. 
 If $\epsilon_1 >0$ is sufficient small, we ensure that $\spt(\mathbf{G}_\gamma) \subset \{ \abs{z}\le \sin(\frac{\theta}{4}) \abs{x} \}$. Furthermore $\Phi$ maps $\Gamma_0$ into $\Gamma_0$. Fix a non-negative smooth function $\eta$ with $\eta=1$ for $t\le \sin(\frac{\theta}{4})$ and $\eta=0$ for $t\ge \sin(\frac{\theta}{2})$. Now we define the smooth diffeomorphism $F(x):= \left(y, z+ \eta\left(\frac{\abs{z}}{\abs{x}}\right) \gamma(y) \right)$. Since $y\mapsto F(y,0)$ is a parametrization of $\mathbf{G}_\gamma$ we have that 
 \[ \psi(x):= \Phi^{-1} \circ F \circ \Phi\]
 maps $\Gamma_0$ onto $\mathbf{G}_\gamma$. Since $F$ satisfies the bounds $(iii)$ we conclude that $\psi$ satisfies similar bounds. Its inverse $\phi=\psi^{-1}$ has the desired properties.
\end{proof}

In the proof of the almost monotonicity formula the following small observation on measures is helpful. 
\begin{lemma}\label{lem.integration of measures}
	Let $\mu$ be a (non-negative) Radon measure on $\R^n$ and $f$ be a $C^1$ function on the interval $[a,b]$; then the following identity holds true
	\begin{equation}\label{eq:integration of measures}
		\int_a^b f(t) \frac{d}{dt}\mu(B_t) \, dt = \int_{B_b\setminus B_a} f(\abs{x}) \, d\mu. 
	\end{equation}
\end{lemma}

\begin{proof}
Note that since the function $t \mapsto \mu(B_t)$ is monotone increasing, it is a BV function and its derivate $\frac{d}{dt}\mu(B_t)$ is a non-negative measure for which the fundamental theorem of calculus holds. Furthermore since the chain rule holds for BV functions, we have that $\frac{d}{dt} (f(t)\mu(B_t)) = f'(t) \mu(B_t) + f(t) \frac{d}{dt} \mu(B_t)$. Rearranging and integrating between $a$ and $b$ gives
\begin{equation}\label{eq. chain for BV}
 \int_a^b f(t) \frac{d}{dt}\mu(B_t) \, dt = f(b)\mu(B_b) - f(a)\mu(B_a) + \int_a^b f'(t) \mu(B_t)\, dt. 
 \end{equation}
Let us consider the last integral. Since $f'$ is continuous we can apply Fubini's theorem and obtain
\begin{align*}
	\int_a^b f'(t) \mu(B_t)\, dt &= \int_{\R^n}\int_a^b f'(t) \mathbf{1}_{B_t}(x)\, dt d\mu = \int_{B_b} \int_{\max(\abs{x},a)}^b f'(t)\, dt d\mu \\
	&= \int_{B_b\setminus B_a} \int_{\abs{x}}^b f'(t)\, dt d\mu + \int_{B_a} \int_{a}^bf'(t)\,dt d\mu \\
	&=f(b) \mu(B_b\setminus B_a) - \int_{B_b\setminus B_a} f(\abs{x}) \, d\mu + (f(b)-f(a))\mu(B_a)\\
	&=f(b) \mu(B_b) - f(a) \mu(B_a) - \int_{B_b\setminus B_a} f(\abs{x}) \, d\mu.
\end{align*}
By combining the above identity with \eqref{eq. chain for BV}, it then follows the validity of \eqref{eq:integration of measures}.
\end{proof}

We are now in position to state and prove the boundary version of the almost monotonicity formula.
\begin{proposition}\label{prop:almost monotonicity formula}
Let $T$ be \emph{almost minimizing} at $x \in \spt(T)\cap \Gamma$ with $C^{1,\alpha_1}$-boundary $\Gamma$. Then there are constants $C_3, r_3, \alpha_3$ depending on $C_0,r_0,\alpha_0,\Gamma$ such that 
\begin{equation}\label{eq:almost monotonicity formula}
e^{C_3r^{\alpha_3}}\frac{\norm{T}(B_r(x))}{r^m} - 	e^{C_3s^{\alpha_3}}\frac{\norm{T}(B_s(x))}{s^m} \ge \int_{B_r(x) \setminus B_s(x)} e^{C_3\abs{z-x}^{\alpha_3}} \frac{\abs{(z-x)^\perp}^2}{2\abs{z-x}^{m+2}} \, d\norm{T}(z)\,.
\end{equation}
\end{proposition}

\begin{proof}
After translation we can assume that $x=0$. Let $\phi$ denote the map constructed in Lemma \ref{lem.straight boundary}. In particular, by the construction of $\phi$ we have 
\[ d\circ \phi(x) = d(x)  \text{ for all } \abs{x}<r_1, \]
where $d(x)=\abs{x}$, as explained. By classical slicing theory we have that, for almost every $t$, the slice $\langle T, d, t \rangle$ is integral, \cite[sec. 28]{Simon} and satisfies 
\[ \langle \phi_\sharp T , d, t \rangle = \phi_\sharp  \langle T , d\circ \phi, t \rangle = \phi_\sharp  \langle T , d, t \rangle\,.\]
Hence we deduce that for a.e. $t<r_1$
\begin{align*} \partial( 0\cone\phi_\sharp\langle T , d, t \rangle) &= \phi_\sharp  \langle T , d, t \rangle + 0\cone \langle \phi_\sharp\partial T, d ,t \rangle=  \phi_\sharp  \langle T , d, t \rangle + 0\cone \langle \Gamma_0, d ,t \rangle\\
&=\phi_\sharp  \langle T , d, t \rangle + \Gamma_0 \res B_t. \end{align*}
This implies that $H:=\phi^{-1}_\sharp ( 0\cone\phi_\sharp\langle T , d, t \rangle)$ is an admissible competitor for $T$ in $B_t$ since 
\[ \partial \phi^{-1}_\sharp ( 0\cone\phi_\sharp\langle T , d, t \rangle) = \langle T,d,t\rangle + \Gamma\res B_t = \partial (T\res B_t). \]
By the \emph{almost minimizing} property of $T$ we deduce that for a.e. $t< \min\{r_0,r_1\}$ that 
\begin{align} \nonumber \norm{T}(B_t(x)) &\le (1+C_0t^{\alpha_0}) (1+ C_1 t^{\alpha_1})^m \norm{  0\cone\phi_\sharp\langle T , d, t \rangle}(B_t)\\ \nonumber &= (1+C_0t^{\alpha_0}) (1+ C_1 t^{\alpha_1})^m \frac{t}{m} \mathbf{M}(\phi_\sharp\langle T , d, t \rangle) \\\nonumber &\le (1+C_0t^{\alpha_0}) (1+ C_1 t^{\alpha_1})^{2m-1} \frac{t}{m} \mathbf{M}(\langle T , d, t \rangle)\\ \label{eq:distributional derivative}& \le (1 + \frac{\alpha_3 C_3}{2m} t^{\alpha_3}) \frac{t}{m} \mathbf{M}(\langle T , d, t \rangle) \,.\end{align}
We have $\abs{\nabla_{\vec{T}}d}(x)=\frac{\abs{x^T}}{\abs{x}}$, where $\nabla_{\vec{T}}d$ denotes the gradient of $d$ along the approximate tangent plane $\vec{T}$ and $^T$ denotes the projection onto the approximate tangent plane. Thus by slicing theory we have for a.e. $0\le a < b$ that
\[ \int_{a}^{b} \mathbf{M}(\langle T,d,t\rangle) \, dt = \int_{B_b\setminus B_a} \frac{\abs{x^T}}{\abs{x}} \, d\norm{T}\,. \]
We will denote the non-negative Radon measure on RHS by $\mu$ i.e. $\int \varphi \,d\mu = \varphi \, \frac{\abs{x^T}}{\abs{x}} \, d\norm{T}$. In particular by the above identity we have for a.e. $t$ that 
\[ \mu'(B_t) = \mathbf{M}(\langle T,d,t\rangle), \]
where $\mu'(B_t)$ denotes the distributional derivative of $t\mapsto \mu(B_t)$. 
Hence we conclude that for a.e. $t \le r_3<\min\{r_0, r_1\}$
\begin{align*}
&\frac{d}{dt}\left( t^{-m} \norm{T}(B_t) \right) = t^{-m}\left( -m \norm{T}(B_t) + \norm{T}'(B_t) \right)	\\
&\ge -\alpha_3 C_3 t^{\alpha_3-1} \frac{\norm{T}(B_t)}{t^m}+ t^{m} \left( - \frac{m}{(1 + \frac{\alpha_3 C_3}{2m} t^{\alpha_3})} \norm{T}(B_t) + \norm{T}'(B_t)\right)\\
&\ge -\alpha_3 C_3 t^{\alpha_3-1} \frac{\norm{T}(B_t)}{t^m} + t^{-m}\left( \norm{T}'(B_t) - \mu'(B_t) \right)\,.
\end{align*}
Since $\frac{\abs{x^T}}{\abs{x}} \le 1$ we have $\norm{T}'(B_t) \ge \mu'(B_t) $ for a.e. $t$. So we conclude that 
\[ \frac{d}{dt} e^{C_3t^{\alpha_3}} \left( t^{-m} \norm{T}(B_t) \right) \ge e^{C_3t^{\alpha_3}} t^{-m}\left( \norm{T}'(B_t) - \mu'(B_t) \right)\,.\]
Integrating the above inequality from $0<s<r<r_3$ and applying Lemma \ref{lem.integration of measures} we conclude 
\begin{align*}
	e^{C_3r^{\alpha_3}}\frac{\norm{T}(B_r(x))}{r^m} - 	e^{C_3s^{\alpha_3}}\frac{\norm{T}(B_s(x))}{s^m} &\ge \int_{B_r\setminus B_s} \frac{e^{C_3\abs{x}^{\alpha_3}}}{\abs{x}^m} \left(1-\frac{\abs{x^T}}{\abs{x}} \right) \, d\norm{T}(z)\\
	&\ge \int_{B_r\setminus B_s} e^{C_3\abs{x}^{\alpha_3}}\frac{\abs{x^\perp}^2}{2\abs{x}^{m+2}}\, d\norm{T}(z)\,.
\end{align*}
\end{proof}

\begin{remark}\label{rem:density}
A classical consequence of the almost monotonicity formula is that the density of the current $T$ exists in $x$
\begin{equation}\label{eq: density of the current}
	\Theta(T,x)= \lim_{r \to 0} \frac{ \norm{T}(B_r(x))}{\omega_m r^m}\,.
\end{equation}
	
\end{remark}

\begin{corollary}\label{cor:continuity of the slices in flat norm}
	Under the same conditions as in Proposition \ref{prop:almost monotonicity formula} the map 
	\[ r \mapsto \frac{ \langle T, d_x, r \rangle }{r^{m-1}} \]
	%where $d(y)=\abs{y-x}$ 
	is continuous in the flat norm on $0< r < r_3$. More precisely, for any $r<s<r_3$, one has 
\begin{align}\label{eq:modulus of continuity}
		&\mathcal{F}\left(\frac{ \langle T, d_x, s \rangle }{s^{m-1}} - \frac{ \langle T, d_x, r \rangle }{r^{m-1}}\right)^2\le 2 \,\bM((i_x)_\sharp T\res B_s\setminus B_r)^2 + C s^{2\alpha_1} \\
		&\nonumber\le C \left(\ln(\frac{s}{r}) e^{C_3s^{\alpha_3}}\frac{\norm{T}(B_s(x))}{s^m}\right)\left(e^{C_3s^{\alpha_3}}\frac{\norm{T}(B_s(x))}{s^m} - 	e^{C_3r^{\alpha_3}}\frac{\norm{T}(B_r(x))}{r^m}\right)\,+ Cs^{2 \alpha_1},
	\end{align}
	where $i_x(y):= \frac{y-x}{\abs{y-x}}$.

\end{corollary}
\begin{proof}
As before we may assume without loss of generality that $x=0$. 
Since $\spt(\langle T, d, r \rangle) \subset \partial B_r$, we have 
\[ i_\sharp \langle T, d, r \rangle = \frac{ \langle T, d, r \rangle }{r^{m-1}} \,.\]
% where $i(x) = \frac{x}{\abs{x}}$. 
We note that, for $r<s<r_3$, 
\[ \langle T, d, s\rangle - \langle T, d, r\rangle = \partial (T\res (B_s\setminus B_r)) - \Gamma \res(B_s \setminus B_r)\,. \]
So we can estimate
\begin{equation}\label{eq:distance flat norm slices} \mathcal{F}\left( i_\sharp \langle T, d, s\rangle - i_\sharp\langle T, d, r\rangle\right) \le \mathbf{M}(i_\sharp T\res(B_s\setminus B_r)) +\mathbf{M}(i_\sharp \Gamma\res(B_s\setminus B_r)).\end{equation}

and
\begin{align*}  &\mathbf{M}(i_\sharp T\res(B_s\setminus B_r)) = \int_{B_s\setminus B_r} \frac{\abs{x^\perp}}{\abs{x}^{m+1}} \, d\norm{T}\\& \le \left( \int_{B_s\setminus B_r} \frac{1}{\abs{x}^m} \,d\norm{T}\right)^{\frac12}\left( \int_{B_s\setminus B_r} \frac{\abs{x^\perp}^2}{\abs{x}^{m+2}} \,d\norm{T}\right)^{\frac12}\,.\end{align*}
The second term is directly bounded by \eqref{eq:almost monotonicity formula}. The first part can be bounded by splitting  $B_s\setminus B_r\subset \bigcup_{l=1}^k B_{2^{-l}s}\setminus B_{2^{-l}s}$ for $l\le C \ln(\frac{s}{r})$ and using the \eqref{eq:almost monotonicity formula} by 
\[	\int_{B_s\setminus B_r} \frac{1}{\abs{x}^m} \,d\norm{T} \le C \ln(\frac{s}{r}) \frac{\norm{T}(B_s)}{s^m}. \]
The second term in \eqref{eq:distance flat norm slices} can be bounded using the regularity of $\Gamma$. By direct computations one obtains that $\abs{x^\perp}\le C \abs{x}^{1+\alpha_1}$ hence 
\begin{align*} \mathbf{M}(i_\sharp \Gamma\res(B_s\setminus B_r))&= \int_{\Gamma\cap(B_s\setminus B_r)} \frac{\abs{x^\perp}}{\abs{x}^{m+1}} \le C \int_{\Gamma\cap(B_s\setminus B_r)} \abs{x}^{\alpha_1-m}\\
&\le C s^{\alpha_1}\,.
\end{align*}
\end{proof}

\section{White's Epiperimetric inequality and its generalizations for boundary points}\label{sec.epiperimetric}
In this section we prove a generalization of White's epiperimetric inequality at boundary points. The argument is very close to White's original one \cite{White} and the argument presented in \cite[Lemma 3.3]{DeLellis}.
 Let us denote by $\Gamma_0$ the line $\R \times \{0\} \in \R^{n+2}$. We recall the characterization of $2$-dimensional area minimizing cones with boundary \cite[Lemma 3.18]{DeLellisDePhilippis}

\begin{lemma}[Characterization of 2 dimensional area minimizing cones with boundary]\label{lem. characterisation of boundary cones}
Let $T_0$ be an integral 2-dimensional locally area-minimizing current in $\R^{2+k}$ with $(\iota_{0,r})_\sharp T_0 = T_0$ for every $r>0$ and $\partial T_0 =  \a{\Gamma_0}$, where $\Gamma_0 = \{(x,y)\in \R^2\times \R^{k}: x_1=|y|=0\}$, Then 
\[
T_0 = \a{\pi^+} + \sum_{i=1}^N \theta_i \a{\pi_i}
\] 
where 
\begin{itemize}
\item[(a)] $\pi^+$ is a closed oriented half-plane;
\item[(b)] the $\pi_i$'s are all oriented $2$-dimensional planes which can only meet at the origin;
\item[(c)] the coefficients $\theta_i$'s are all natural numbers;
\item[(d)] if $\pi^+\cap \pi_i\neq \{0\}$, then $\pi^+\subset \pi_i$ and they have the same orientation.
\end{itemize}
\end{lemma}

\begin{lemma}\label{lem.boundary epiperimetric I}
Let $S\in \mathbf{I}_2(\R^{n+2})$ be an area-minimizing cone with $\partial S=\Gamma_0$. There exists a constant $\epsilon_1=\epsilon_1(S)>0$ with the property: if $R: = \langle S, d, 1\rangle$, where $d(x) = \abs{x}$ and $Z \in \mathbf{I}_2(\partial B_1)$ with $\partial Z = \partial R = \a{e_1}- \a{-e_1}$ and
\begin{itemize}
\item[(a1)] $\mathcal{F}(Z-R) \le \epsilon_1$;
\item[(a2)] $\mathbf{M}(Z)- \mathbf{M}(R) \le \epsilon_1$;
\item[(a3)] $\dist(\spt(Z), \spt(R))\le \epsilon_1$	
%{\color{blue} I am not sure but I think it would be sufficient to have $\dist( x, \spt(R))\le \epsilon_1 \forall x \in \spt(Z)$}
\end{itemize}
then there exists $H \in \mathbf{I}_2(B_1)$ sucht that $\partial H = \partial (S\res B_1)$ and 
\begin{equation}\label{eq:epiperimetric inequality} \norm{H}(B_1) - \norm{S}(B_1) \le (1- \epsilon_1) \left( \norm{0 \cone Z}(B_1) - \norm{S}(B_1) \right). \end{equation}
\end{lemma}
Before we will prove this "new" boundary adaption of the "classical" version, we want to indicate how to use it to close the argument. We can almost verbatim follow \cite{DeLellis}.
We use a compactness argument to generalize the above lemma to:
\begin{proposition}\label{prop:boundary epi-perimetric}
	Let $S\in \mathbf{I}_2(\R^{n+2})$ be an area-minimizing cone with $\partial S=\Gamma_0$. There is a constant $\epsilon_4>0$ with the property: $T$ is an \emph{almost minimizing} cone in a neighbourhood $U$ of $0$ as in definition \ref{def.almost minimzing}. There are positive constants $C_4, r_4, \alpha_4$ such that if
	\begin{itemize}
	\item[(a1)]	$\bd_{B_2}(T_{r},S)< \epsilon_4$;
	%\item[(a2)] $\norm{T_r}(B_3)< C_3 3^2$;
	\item[(a2)] $r< \epsilon_4$;
	\item[(a3)] $\langle T_r, d, 1\rangle \in \mathbf{I}_1(\R^{n+2})$ 
	\end{itemize}
then 
 \begin{equation}\label{eq:epiperimetric inequality 2} \norm{T_r}(B_1) - \norm{S}(B_1) \le (1- \epsilon_4) \left( \norm{0 \cone \langle T_r, d,1 \rangle}(B_1) - \norm{0 \cone \langle S, d,1 \rangle}(B_1) \right)+ C_4 r^{\alpha_4}. \end{equation}
\end{proposition}

\begin{proof}
%{\color{blue} I am not sure if we actually need that first part - since in (a2) of lemma \ref{lem.boundary epiperimetric I} we don't have absolut values. }{\color{magenta} If  $\norm{0 \cone \langle T_r, d,1 \rangle}(B_1) - \norm{0 \cone \langle S, d,1 \rangle}(B_1) \le 0$ \eqref{eq:epiperimetric inequality 2} holds true as a consequence of the \emph{almost minimising} property and the almost monotonicity since 
%\begin{align*}
%	&\norm{T_r}(B_1) - \norm{S}(B_1) \le \frac{1}{1+C_0 r^{\alpha_0}} \norm{T_r}(B_1) - \norm{S}(B_1) + \frac{C_0r^{\alpha_0}}{1+C_0r^{\alpha_0}} \norm{T_r}(B_1)\\
%	&\le \norm{0 \cone \langle T_r, d,1 \rangle}(B_1) - \norm{0 \cone \langle S, d,1 \rangle}(B_1) + \frac{C_0r^{\alpha_0}}{1+C_0r^{\alpha_0}} \norm{T_r}(B_1)\\
%	&\le (1- \epsilon_3) \left( \norm{0 \cone \langle T_r, d,1 \rangle}(B_1) - \norm{0 \cone \langle S, d,1 \rangle}(B_1) \right)+ 2C_3 r^{\alpha_0}.
%\end{align*}
%Where we used in the second inequality that $\partial 0 \cone \langle T_r, d,1 \rangle$ is an admissible competitor for $T_r$ since the boundary is straight and $\norm{S}(B_1)= \norm{0 \cone \langle S, d,1 \rangle}(B_1)$ since $S$ is a cone. Finally we used in the last line the almost monotonicity. }

We argue by contradiction. So we find a sequence $T^k$ being \emph{almost minimizing} in $U$ with constants $C_0, r_0, \alpha_0>0$,  $r_k \downarrow 0$ ( i.e. $3r_k < r_0$ ) such that
\[ \bd_{B_2}(T_k, S)< 2^{-k} , \quad \langle T_k, d, 1\rangle \in \mathbf{I}_1(\R^{n+2}), \]
 where $T_k:=T_{r_k}^k$, but failing to satisfy \eqref{eq:epiperimetric inequality 2}. Note that $T_k$ is now \emph{almost minimizing} in $B_4$, for each $k$ sufficiently large.
 
By assumption we have $T_k \to S$ in the flat metric topology on $B_2$ which is equivalent to $T_k \rightharpoonup S$ in $B_2$, \cite[Theorem 6.7.1]{Simon} . Hence by the lower semicontinuity of mass 
\[ \norm{S}(B_r) \le \liminf_{k \to \infty} \norm{T_k}(B_r). \]
To each $k$ we may fix $R_k, Q_k$ with $T_k-S = R_k + \partial Q_k$ in $B_3$ and $\norm{R_k}(B_2)+\norm{Q_k}(B_2) < 2^{1-k}$. Since $\sum_{k} \norm{R_k}(B_2)+\norm{Q_k}(B_2) < \infty$, given $r<2$ we can find $r<s <2$ such that $\limsup_{k} \mathbf{M}( R_k \res B_s ) + \mathbf{M}(\langle Q_k, d, s\rangle )=0$. Hence we have 
\[ T_k\res B_s = S\res B_s + R_k\res B_s - \langle Q_k, d, s \rangle + \partial (Q_k \res B_s) \,.\]
Thus by the almost minimality of $T_k$ in $0$ we deduce that 
\[ \norm{T_k}(B_s) \le (1+ C_0 r_k^{\alpha_0}) \left(\norm{S}(B_s) + \norm{R_k}(B_s) + \mathbf{M}(\langle Q_k, d, s\rangle )\right)\,.\]
So $\limsup_{k\to \infty}\norm{T_k}(B_s)\le \norm{S}(B_s)$ and in combination with the above that $\norm{T_k} \to \norm{S}$ on $B_2$ in the sense of measures. Since $S$ is a cone, we have  $\norm{S}(\partial B_r)=0$, for all $r$, and therefore $\lim_{k\to \infty} \norm{T_k}(B_r) = \norm{S}(B_r)$ for all $r<2$. Now the failure of \eqref{eq:epiperimetric inequality 2} implies that 
\[\lim_{k\to \infty} \mathbf{M}(\langle T_k, d, 1\rangle) - \mathbf{M}(\langle S, d, 1 \rangle) \le 0 \,.\]
The almost monotonicity formula in the interior, \cite[Proposition 2.1]{DeLellis}, combined with \eqref{eq:almost monotonicity formula} implies that $\norm{T_k}(B_r(x)) \ge c r^2$ for all $x \in B_2$ and $0<r<1$. This density lower bound entails, by standard arguments, that $\spt(T_k)$ converges to $\spt(S)$ in the Hausdorff sense on $B_2$.
Let $\phi$ be the map described in Lemma \ref{lem.straight boundary}, and set \[ Z_k := (\phi_k)_\sharp \langle T_k, d, 1 \rangle \in \mathbf{I}_1(\partial B_1)\,,\]
where $\phi_k(x) = \frac{1}{r_k} \phi(r_k x)$. Recall that, thanks to the properties of $\phi$, we have \[l_k:=\norm{\phi_k(x) - x}_{C^1(B_2)}\le C_1 r_k^{\alpha_1} \to 0\,.\] 
 In particular this implies that $\spt(Z_k) \to \spt(S)$ in the Hausdorff sense on $B_2$ and 
 \[  (1- l_k)^2 \,\mathbf{M}(\langle T_k, d, 1\rangle) \le \mathbf{M}(\langle Z_k, d, 1\rangle) \le (1+ l_k)^2 \,\mathbf{M}(\langle T_k, d, 1\rangle)\,.  \]
So (a2) and (a3) in Lemma \ref{lem.boundary epiperimetric I} hold, for sufficiently large $k$. It remains to show (a1). As noticed before, we have $\limsup_{k} \mathbf{M}(\langle R_k, d, s\rangle ) + \mathbf{M}(\langle Q_k, d, s\rangle )=0$ for almost every $s<2$. Since 
\[ \langle T_k, d, s \rangle - \langle S, d, s \rangle = \langle R_k, d, s, \rangle - \partial \langle Q_k, d, s \rangle \]
we deduce that $\mathcal{F}( \langle T_k, d, s \rangle - \langle S, d, s \rangle )=0$ for a.e. $s<2$. Furthermore we have for any $r<s$
\[ \langle T_k, d, s \rangle - \langle T_k, d, r \rangle = \partial T_k\res B_r\setminus B_s - \Gamma_{r_k} \res B_r\setminus B_s\]
and therefore $\mathcal{F}(\langle T_k, d, s \rangle - \langle T_k, d, r \rangle) \le \norm{T_k}(B_s\setminus B_r) + \norm{\Gamma_{r_k}}(B_s\setminus B_r)$. Since $\norm{T_k}\to \norm{S}$, $\Gamma_{r_k} \to \Gamma_0$ and $\norm{S}(\partial B_r)+ \norm{\Gamma_0}(\partial B_r)=0$ for all $r$ we conclude that
\begin{equation}\label{eq:flat norm of slice} \lim_{k\to \infty} \mathcal{F}(\langle T_k, d, 1 \rangle - \langle S, d, 1\rangle) =0 \,.\end{equation}
We may use the homotopy formula to estimate the flat distance between $Z_k$ and $\langle T_k, d, 1 \rangle$: set $h_k(t,x):= (1-t) x + t \phi_k(x)$, so $\norm{h_k(t,x) - x}_{C^1} \to 0$ and 
\begin{align*} Z_k - \langle T_k, d, 1 \rangle &= \partial (h_k)_\sharp\left( [0,1]\times \langle T_k, d, 1 \rangle\right) - (h_k)_\sharp \left([0,1]\times \langle \Gamma_{r_k}, d, 1 \rangle\right)\\
&= \partial \tilde{Q}_k - \tilde{R}_k.\end{align*}
Since in particular $\norm{\partial_t h_k(t,x)}_{C^0(B_2)}=\norm{\phi_k(x)-x}_{C^0(B_2)} \to 0$, we have 
\begin{align*}
\bM(\tilde{Q}_k) &\le \int_0^1 \int \abs{\partial_t h_k} 	\, d\norm{\langle T_k, d, 1 \rangle}\, dt \to 0\,, \\
\bM(\tilde{R}_k) &\le \int_0^1 \int \abs{\partial_t h_k} 	\, d\norm{\langle \Gamma_{r_k}, d, 1 \rangle}\, dt \to 0\,.
\end{align*}
We combine it now with \eqref{eq:flat norm of slice} that (a1) holds for large $k$ since $\mathcal{F}(Z_k - S)\to 0$. 

For each $k$ large enough, we may then apply Lemma \ref{lem.boundary epiperimetric I} and obtain $H_k \in \mathbf{I}_2(B_1)$ satisfying \eqref{eq:epiperimetric inequality}. By construction of $\phi_k$ we deduce that $(\phi_k^{-1})_\sharp H_k$ is an admissible competitor for $T_k$ in $B_1$ since $ \partial (T_k - (\phi_k^{-1})_\sharp H_k)=0$. Using $\frac{(1+C_0 r_k^{\alpha_0})(1+C_1 r_k^{\alpha_1})-1}{(1+C_0 r_k^{\alpha_0})(1+C_1 r_k^{\alpha_1})} \le \frac{C_4}{4\norm{S}(B_1)} r_k^{\alpha_4}$ and again $\lim_{k\to\infty}\norm{T_k}(B_1)\to \norm{S}(B_1)$ we conclude 
\begin{align*}
&\norm{T_k}(B_1) - \norm{S}(B_1) \le \frac{\norm{T_k}(B_1)}{(1+C_0 r_k^{\alpha_0})(1+C_1 r_k^{\alpha_1})} - \norm{S}(B_1) + C_4 r_k^{\alpha_4}\frac{\norm{T_k}(B_1)}{4\norm{S}(B_1)} \\
&\le \frac{\norm{(\phi^{-1}_k)_\sharp H_k}(B_1)}{1+C_1 r_k^{\alpha_1}}-\norm{S}(B_1) + \frac{C_4 r_k^{\alpha_4}}{2}\le \norm{H_k}(B_1) -\norm{S}(B_1) + \frac{C_4 r_k^{\alpha_4}}{2} \\
&\le (1- \epsilon_1) \left( \norm{0 \cone Z_k}(B_1) - \norm{0 \cone \langle S, d,1 \rangle}(B_1) \right)+ \frac{C_4 r_k^{\alpha_4}}{2}\\
&\le (1- \epsilon_1) \left( \norm{0 \cone \langle T_k, d,1 \rangle}(B_1) - \norm{0 \cone \langle S, d,1 \rangle}(B_1) \right)+ C_4 r_k^{\alpha_4}\,.
\end{align*}
 This is a contradiction if $\epsilon_4<\epsilon_1$ for sufficiently large $k$. 
\end{proof}

%\begin{corollary}\label{cor.boundary epiperimetric II}
%	Let $S, \Gamma_0, R=\langle S, d, 1\rangle$ as in lemma \ref{lem.boundary epiperimetric I}. There exits $0<\epsilon_2=\epsilon_2(S) <, \epsilon_1$ with the property: if $Z \in \mathbf{I}_2(\partial B_1)$ and
%\begin{itemize}
%\item[(a1)] $\mathcal{F}(Z-R) \le \epsilon_2$;
%\item[(a2)] $\mathbf{M}(Z)- \mathbf{M}(R) \le \epsilon_2$;
%\item[(a3)] $\dist(\spt(Z), \spt(R))\le \epsilon_1$	
%{\color{blue} I am not sure but I think it would be sufficient to have $\dist( x, \spt(R))\le \epsilon_1 \forall x \in \spt(Z)$}
%\item[(a4)] $\mathcal{F}(\partial Z - \partial R) \le \epsilon_2$
%\end{itemize}
%then there exists $H \in \mathbf{I}_2(B_1)$ sucht that $\partial H = \partial (Z\res B_1)$ and 
%\begin{equation}\label{eq:epiperimetric inequality II} \norm{H}(B_1) - \norm{S}(B_1) \le (1- \epsilon_2) \left( \norm{0 \cone Z}(B_1) - \norm{S}(B_1) \right)+ C \mathcal{F}(\partial Z - \partial R). \end{equation}
%\end{corollary}

\begin{proof}[Proof of Lemma \ref{lem.boundary epiperimetric I}]
By the previous classification lemma, \ref{lem. characterisation of boundary cones}, we have that the support of $R$ is either the disjoint union of $N$ equatorial circles of $\partial B_1$ and one half circle, or the disjoint union of $N+1$ equatorial circles of $\partial B_1$ depending which option in (d) applies: $R=R_0 + \sum_{i=1}^N R_i$, where $\spt(R_i) \subset \pi_i \cap \partial B_1$ for pairwise disjoint $2$-dimensional planes $\{\pi_i\}_{i=0}^N$. If $\epsilon_1>0$ in (a3) is sufficient small with respect to the distance of the equatorial circles/ half-circle $Z$ splits as well: in both cases we get $Z= Z_0 + \sum_{i=1}^N Z_i$ with $\partial Z_i =0 $ for $i\ge 1$ and $\partial Z_0 = \partial R=\partial R_0$. 
Hence we can consider each pair $Z_i, R_i$ separately. The cases $i\ge 1$ correspond to the interior situation, i.e. \cite[proof of Lemma 3.3]{DeLellis}. 
Thus, without loss of generality, from now on we assume that 
\[S = Q \a{\pi} + \a{\pi^+} \text{ and } R = Q \a{\Gamma_0} + \a{\Gamma_0^+},\]
where $\Gamma_0$ is the oriented equatorial circle $\pi \cap \partial B_1$ and $\Gamma_0^+$ the oriented half circle
$\Gamma_0\cap \pi^+$.

\emph{Step 1: Reduction to Lipschitz winding curves}:
Let us introduce some notation. Given a two dimensional subspace $\tau$ we will denote with $\bp_\tau$ the orthogonal projection of $\R^{n+2}$ onto $\tau$. We will denote $\bC_r(\tau)$ the cylinder $\bp_{\tau}^{-1}( B_r\cap \tau)$.\\
Let $\bP(x):=\frac{x}{\abs{x}}$ be the projection onto $\partial B_1$. If we restrict $\bP$ to $\partial \bC_1(\tau)$, it becomes an invertible diffeomorphism from $\partial \bC_1(\tau)$ onto $\partial B_1 \setminus \tau^\perp$. We will denote its inverse by $\bP_\tau^{-1}$. A direct computation gives that for $x \in \partial B_1$ and $\dist(x, \tau)< \epsilon_1$ that $\operatorname{Lip}(\bP(x)_\tau^{-1}) \le (1+ C\dist(x,\tau))$.
Furthermore in the proof we will use the general notion of excess.  Recall that the \emph{cylindrical excess} of any current $T$ with respect to the plane $\tau$ with orientation $\vec{\tau}$ is given by 
\[ E(T, \tau):= \frac12 \int_{\mathbf{C}_1(\tau)} \abs{\vec{T}(x) - \vec{\tau}}^2 \, d \norm{T}(x)= \mathbf{M}(T\res\mathbf{C}_1) - \mathbf{M}({\bp_\tau}_\sharp T \res \mathbf{C}_1(\tau)). \]
Given any curve $Z\in \partial B_1$ we define its excess to be 
\[ E(Z, \tau):= E( (0\cone Z)_\infty,\tau).\]
Generally the cylindrical excess of $T$ is then given by $ E(T):= \min_\tau E(T, \tau) $
and equivalent for curves $Z$ by $E(Z):= \min_\tau E(Z, \tau)$. 
Fur currents $T$ with $\partial T = \Gamma_0$ or curves with $\partial Z = \partial R$ we let 
\begin{equation}\label{eq:boundary excess}
E^\flat(T)= \min_{\substack{\tau\\ \Gamma_0 \subset \tau}} E(T,\tau) \text{ and } 	E^\flat(Z)= \min_{\substack{\tau\\ \Gamma_0 \subset \tau}} E(Z,\tau).
\end{equation}
Note that this is well defined since if $\partial Z = \partial R$ we have $\partial (0\cone Z)_\infty = \Gamma_0$. 
It is simple to see that under the assumptions (a1) - (a3) any minimum point $\tau$ for $(0\cone Z)_\infty$ in $E^\flat$ must be close to $\pi$. \\

We may decompose $Z$ into its indecomposable components, see \cite[4.2.25]{Federer} i.e.
\[ Z = Z_0 + \sum_{i=1}^N Z_i, \]
where each $Z_i$ is a closed Lipschitz curve for $i\ge 1$ and $\partial Z_0 = \partial R$ and $\sum_{i=0}^N \mathbf{M}(Z_i) = \mathbf{M}(Z)$. If $P(x):= \frac{\bp_\pi(x)}{\abs{\bp_\pi(x)}}$ then we have by the constancy theorem 
\[ P_\sharp Z = P_\sharp Z_0 + \sum_{i=1}^N P_\sharp Z_i = \theta_0 \a{\Gamma_0} + \a{\Gamma_0^+} + \sum_{i=1}^N \theta_i \a{\Gamma_0} \]
where $\theta_i \in \Z$ for all $i$. But due to (a1) and (a2) we have $\theta_i \ge 0$ for all $i$ and $\sum_{i =0}^N \theta_i = Q$. In particular this implies that each component $Z_i$ satisfies (a2) and (a3).
It is sufficient to prove \eqref{eq:epiperimetric inequality} for each $Z_i$ separately and sum it over $i$. \\ 
As already noted in \cite{DeLellis}, if $\theta_i=0$ for some $i\ge 1$, we can use the isoperimetric inequality to prove \eqref{eq:epiperimetric inequality} i.e. let $H$ such that $\partial H = Z_i$  and 
\[ \norm{H}(B_1) \le C (\bM(Z_i))^2 \le C \epsilon_1 \bM(Z_i) \le C \epsilon_1 \frac12 \norm{0 \cone Z_i}(B_1). \] 
It follows by a standard argument \footnote{ 
If $\epsilon_1 \downarrow 0$ in (a1) we must have $Z_i \rightharpoonup \theta_i \a{\Gamma}_0$ for $i\ge 1$ and $Z_0 \rightharpoonup \theta_0 \a{\Gamma_0} + \a{\Gamma_0^+}$ in the sense of currents. But since weak convergence implies flat norm convergence we may assume choosing $\epsilon_1>0$ sufficient small that each component itself satisfies (a1).} that each $Z_i$ satisfies as well (a1) with $R$ replaced by $\theta_i \a{\Gamma_0}$ for $i \ge 0$ and $R= \theta_0 \a{\Gamma_0} + \a{\Gamma_0^+}$ for $i=0$ and $\epsilon_1'>0$ in place of $\epsilon_1$, where $\epsilon_1' \downarrow 0$ as $\epsilon_1 \downarrow 0$.  In summary we can assume without loss of generality that in addition to (a1) - (a3) we have either 
\begin{itemize}
\item[(a4)] $R= \theta \a{\Gamma_0}$ for some integer $\theta>0$. 
\item[(a5)] $Z= X_\sharp \a{[0, \bM(Z)]}$, where $X: [0, \bM(Z)] \to \partial B_1$ Lipschtiz and $X(0)=X(\bM(Z))$. 
\item[(a6)] if $E(Z, \tau)= E(Z)$ then $E(Z, \tau)\le \epsilon$, ${\bp_\tau}_\sharp {\bP_\tau^{-1}}_\sharp Z= R$.
\end{itemize}
or we have 
\begin{itemize}
\item[(a7)]  $R= \theta_0 \a{\Gamma_0} + \a{\Gamma_0^+}$ for some integer $\theta_0\ge 0$. 
\item[(a8)] $Z= X_\sharp \a{[0, \bM(Z)]}$, where $X: [0, \bM(Z)] \to \partial B_1$ Lipschtiz and $X(0)= - e_1, X(\bM(Z))=e_1$. 
\item[(a9)] if $E^\flat(Z, \tau)= E^\flat(Z)$ then $E^\flat(Z, \tau)\le \epsilon$, ${\bp_\tau}_\sharp {\bP_\tau^{-1}}_\sharp Z= R $.
\end{itemize}
The first case, (a4) - (a6), corresponds to the interior situation and had been proven in \cite[Lemma 3.3]{DeLellis}. It remains to consider the second case, (a7) - (a9)\footnote{The interior situation is proven in a similar manner and the reader may do the obvious modifications from the boundary to the interior case.}.

We set $X_1:= \bP_\tau^{-1} \circ X: [0, \bM(Z)] \to \partial \bC_1(\tau)$. Note that if $Z_1={X_1}_\sharp\a{[0,\bM(Z)]}$ then  $(0\cone Z)_\infty = (0 \cone Z_1)_\infty$. Hence we have $E^\flat(Z) =E^\flat(\hat{Z})$ and a straight forward calculation reveals if $X_1(t) = (e^{i\theta_1(t)}, y_1(t)) \in \tau \times \tau^\perp; x_1(t)=(\theta_1(t),y_1(t))$ then 
\[  E:=E(Z_1, \tau) = \int_0^{\bM(Z)} \abs{\dot{x}_1}_h - \dot{\theta}_1. \]
We may extend $X_1$ to $[-\bM(Z), 0]$ by reflecting it along $\Gamma_0$ i.e. $X_1(-t) = (e^{-i\theta_1(-t)}, - y_1(-t))$. As a consequence we may apply Lemma \ref{lem:Lipschitz approximation} and obtain for any $\delta>0$ a function $y_2: [-(2\theta_0 +1)\pi, (2\theta_0 +1)\pi] \to \tau^\perp$ such that
\begin{itemize}
\item[(b1)] $Z_2={X_2}_\sharp\a{[0, (2\theta_0 +1)\pi]}$ where $X_2(t)= (e^{it}, y_2(t))$, $y_2: [-(2\theta_0 +1)\pi, (2\theta_0 +1)\pi] \to \tau^\perp$, $\norm{y_2}_\infty \le \epsilon$, $\operatorname{Lip}(y_2) \le C(\delta + \frac{\sqrt{E\epsilon}}{\delta})$ and \[X_2(0)=-e_1, X_2((2\theta+1)\pi) = -e_1.\] 
\item[(b2)] $\bM( Z_1 -Z_2) \le C \frac{E}{\delta^2}$;
\item[(b3)] $E^\flat(Z_2,\tau) \le E^\flat(Z_1,\tau)$;	
\end{itemize}
Next observe that if $\tau_2$ minimizes $E^\flat( (0\cone Z_2)_\infty, \tau_2)$ then we have using that by (b1) we have $\norm{(0\cone Z_2)_\infty}(B_1) \ge \frac{\pi}{4}$ and so
\[ \abs{\tau- \tau_2}^2 \le C E^\flat( (0\cone Z_2)_\infty, \tau)< C\epsilon\]
for some geometric constant $C$. In particular if we combine it with (b1), we conclude that the Lipschitz constant of $\bP_{\tau_2}^{-1} \circ \bP: \partial \bC_1(\tau) \to \partial \bC_1(\tau_2)$ is bounded by $(1+C \epsilon)$ hence $X_3:= \bP_{\tau_2}^{-1} \circ \bP \circ X_2$ is a Lipschitz curve in $\partial \bC_{1}(\tau_2)$ which still can be parametrised over $\partial B_1 \cap \tau_2$. After a rotation we may assume that $\tau_2 = \R^2 \times \{0\}$. In particular we have $X_3(t)=(e^{it}, y_3(t)), t \in [0, (2\theta_0 +1) \pi]$ with $\norm{y_3}_\infty \le C\epsilon, \operatorname{Lip}(y_3) \le C\epsilon$. Furthermore if $Z_3 ={X_3}_\sharp \a{[0,(2\theta_0+1)\pi}$ we have 
\begin{itemize}
\item[(c1)] $(\bP_{\tau_2}^{-1} \circ \bP)_\sharp Z_2= Z_3$ and $(0\cone Z_2)_\infty = (0\cone Z_3)_\infty$,
\item[(c2)] $\bM( (\bP_{\tau_2}^{-1} \circ \bP)_\sharp Z_1- Z_3) \le C \bM(Z_1-Z_2) \le C E $.
\end{itemize}
Expanding $y_3$ in Fourier series, we have $y_3(t)= \sum_{k=1}^\infty a_k \sin(\frac{k}{2\theta_0 +1}t)$. Let us define 
\begin{align*}
f(r,t)&:= \sum_{k =1}^\infty r a_k \sin(\frac{k}{2\theta_0 +1}t),\\
h(r,t)&:= \sum_{k=1}^\infty r^{\frac{k}{2\theta_0+1}} a_k \sin(\frac{k}{2\theta_0 +1}t),
\end{align*}
for $(r,t) \in \Omega:=]0,1[\times ]0, (2\theta_0+1)\pi[$. In particular the graph of $f$ corresponds to $(0\cone Z_3)$. 
We define the linear map $l_0(x,y)=a_{2\theta_0+1} y$ from $\tau_2 \to \tau_2^\perp$ and denote by $\tau_l$ the associated plane. Note that $\Gamma_0\subset \tau_2$ and since $\abs{a_{2\theta_0+1}} \le \epsilon$ we have $\abs{\tau_2 - \tau_l} \le C\epsilon$. By elementary consideration\footnote{ if $A,B: \R^2 \to \R^2\perp$ linear maps given, and $\tau_A,\tau_B$ denote the associated planes i.e. $\tau_A = \frac{e_1 + Ae_1 \wedge e_2 + Ae_2}{\abs{e_1 + Ae_1 \wedge e_2 + Ae_2}}$ then one has 
\[ \frac12 \abs{\tau_A - \tau_B}^2 = \frac12 \norm{A- B}^2 + O(\norm{A}^2\norm{A-B}^2). \]} there is a constant $0<c<1$ such that 
\begin{equation}\label{eq:energy to excess} \abs{Df(r,t) - Dl_0}^2 \ge c \,\abs{ \vec{T}(x) - \tau_l}^2, \end{equation}
where $T(x)$ denotes the oriented tangent plane of the graph of $f$ at the point $(re^{it}, f(r,t))$. A direct computation using the Fourier expansion of $f$ and $h$ shows that there is a constant $c_0=c_0(\theta_0)$, compare with \cite[Proposition 2.4]{White}, that $\int_\Omega \abs{Df}^2 - \int_{\Omega} \abs{Dh}^2 \ge c_0 \int_\Omega \abs{Df - Dl_0}^2$. Combining it with \eqref{eq:energy to excess} and (b3) we conclude that
\begin{equation}\label{eq:bound1} \int_\Omega \abs{Df}^2 - \int_{\Omega} \abs{Dh}^2 \ge c_0 E^\flat(Z_2, \tau_l) \ge c_0 E^\flat(Z_2, \tau_2).\end{equation}
Hence, if $S_1$ is the graph of $h$ i.e. $S_1 = \mathbf{G}_h\sharp \a{\Omega}$, we conclude by using the expansion for the area that 
\begin{align}\label{eq:bound2}
E(S_1,\bC_1(\tau_2))&\le (1+\epsilon) \frac12 \int_\Omega \abs{Dg}^2 \le (1+\epsilon) \frac12	\int_\Omega \abs{Df}^2 - c_0 E^\flat(Z_2, \tau_2)\nonumber\\
&\le ((1+\epsilon)^2 - c_0) E^\flat(Z_2, \tau_2) \le ((1+\epsilon)^2 - c_0) E^\flat(Z_1, \tau_2) .
\end{align}
We can use the isoperimetric inequality to find a current $S_2\subset \bC_1(\tau_2)$ such that $\partial S_2 = (\bP_{\tau_2}^{-1} \circ \bP)_\sharp (Z_1 - Z_3)$ and 
\[ \bM(S_2) \le C \bM((\bP_{\tau_2}^{-1} \circ \bP)_\sharp (Z_1 - Z_3))^2 \stackrel{(c2)}{\le} C E^2 \le CE E(Z_1,\tau_2). \]
Now we define the competitor current for $\rho = 1- C\epsilon$ fixed as
\[ H:= \rho( S_1+S_2) + (0\cone Z_1)_\infty\res B_1 \setminus \bC_\rho(\tau_2)\]
Hence we have
\begin{align*}
\bM(H)&-\bM((0\cone Z_1)_\infty \res B_1\setminus\bC_\rho(\tau_2))= \bM(H\res \bC_\rho(\tau_2))\\
&\le \rho^2(2\theta_0+1)\pi+ \rho^2 E(S_1,\bC_1(\tau_2)) + \rho^2\bM(S_2)\\
& \le \rho^2(2\theta_0+1)\pi+ \rho^2 ((1+\epsilon)^2 +CE - c_0) E(Z_1,\tau_2).
\end{align*}
We may choose $\epsilon>0$ small enough such that $((1+\epsilon)^2 +CE - c_0) \le (1-\frac{c_0}{2})$. Then we may conclude by appealing to the following computation
\begin{align*}
&\bM(H)-\bM((0\cone Z_1)_\infty\res B_1)\\& \le \rho^2(1	-\frac{c_0}{2})E(Z_1,\tau_2) - \left(\bM((0\cone Z_1)_\infty\res \bC_\rho(\tau_2))- (2\theta_0 +1)\pi\rho^2\right)\\
&= \rho^2(1	-\frac{c_0}{2})E(Z_1,\tau_2) - \rho^2E(Z_1,\tau_2) = \rho^2 \frac{c_0}{2} E(Z_1,\tau_2).
\end{align*}
This proofs \eqref{eq:epiperimetric inequality} because
\begin{align*} \rho^2 E(Z_1,\tau_2)&= \int_{\bC_\rho(\tau_2)} \abs{\vec{T}(x)-\tau_2}^2 d \norm{(0\cone Z_1)}\\ &\le  \int_{B_1} \abs{\vec{T}(x)-\tau_2}^2 d \norm{(0\cone Z_1)}\\&
\le \norm{(0 \cone Z_1)_\infty}(B_1) - \norm{(0\cone R)_\infty}(B_1).
\end{align*}
\end{proof}

\section{Boundary tangent cones and proof of Theorem \ref{thm.uniqueness}}\label{sec.final}
In fact, in this section, we will prove the following more accurate version of Theorem \ref{thm.uniqueness}:
\begin{theorem}\label{thm.uniqueness precise}
	Let $T \in \mathbf{I}_2(\R^{2+n})$ \emph{almost minimizing} in a neighbourhood $U$ of $x$. Assume further that $\partial T\res U = \Gamma \res U$ for a $C^{1,\alpha_1}$-boundary $\Gamma$ and $x \in \Gamma$. Then there are $2$-dimensional planes $\{\pi_i\}_i^N, \pi$ intersecting only in $0$, natural numbers $\{ \theta_i\}_{i=1}^N$, $Q \in \N \cup \{0\}$ such that if we set
	\[ S := (Q \a{\pi} + \a{\pi^+}) + \sum_{i=1}^N \theta_i \a{\pi_i} \],
 we have, for $r< \epsilon_4$,
\begin{align}
\label{eq:flat metric convergence} &\bd_{B_1}(T_{x,r},S) \le C_5\, r^{\alpha_5},\\
\label{eq: Haussdorff converegence} &\dist(\spt(T_{x,r}), \spt(S)) = \frac1r \dist(\spt(T)-x, \spt(S))\le C_5\, r^{\frac{\alpha_5}{3}}\,.
\end{align}
Moreover, there are currents $T^i \in \mathbb{I}^2(B_{\epsilon_4}(x)), i=0, \dotsc, N$ such that 
\begin{itemize}
\item[(i)] $\partial T^0 \res B_{\epsilon_4}(x) = \partial \Gamma\res B_{\epsilon_4}(x)	$,  $\partial T^i \res B_{\epsilon_4}(x) = 0 $ for $i\ge1$\\ and each $T^j$ is \emph{almost minimizing} in $ B_{\epsilon_4}(x)$;
\item[(ii)] $T\res B_{\epsilon_4}(x) = \sum_{i=0}^N T^i\res B_{\epsilon_4}(x)$ and $\norm{T}(B_{\epsilon_4}(x)) = \sum_{i=0}^N \norm{T^i}(B_{\epsilon_4}(x))$;
\item[(iii)] $(Q\a{\pi} + \a{\pi^+})$ is the unique tangent cone of $T^0$ and $\theta_i \a{\pi_i}$ is the unique tangent cone of $T^i$ at $x$.
\end{itemize}

\end{theorem}

\begin{proof}[Proof of Theorem \ref{thm.uniqueness precise}] After translation we may assume without loss of generality that $x=0$.

\emph{Step 1: Blowup and compactness of the set of tangent cones.} Due to the almost monotonicity, \eqref{eq:almost monotonicity formula}, the family $\{T_r\}_{r>0}$ is uniformly bounded in mass on every compact subset $K \subset \R^{2+n}$ i.e. $\limsup_{r\to 0} \norm{T_r}(K) < \infty$. In particular for any sequence $r_k\to 0$ we apply the compactness theorem of integral currents,\cite[Theorem 6.3.3 \& 6.8.2]{Simon} , to extract a subsequence (not relabelled) $r_k$ such that $T_{r_k}$ converges in flat norm to an integral current $S$ i.e. 
\begin{equation}\label{eq:flat norm convergence of blow-up}
\bd_{B_R}(T_{r_k}, S) \to 0 \text{ for all } R>0.
\end{equation} We will show now that 
\begin{itemize}
\item[(i)] $S$ is $2$-dimensional area minimizing cone with $\partial S= \Gamma_0$;
\item[(ii)]	$\norm{T_{r_k}}(B_R) \to \norm{S}(B_R)$.
\end{itemize}
 We will call $S$ a tangent of $T$ in $0$.

Since $\partial  T_{r_k} = \Gamma_{r_k}$ and after a possible rotation $\Gamma_{r_k} \to \Gamma_0$ in $C^{1, \beta}$ for all $\beta < \alpha_1$, we deduce that $\partial S= \Gamma_0$. The fact that $S$ is locally area minimizing follows by the lower semicontinuity of the mass and the almost minimality in $B_R$ of the currents $T_{r_k}$.
 To show (ii), we follow the arguments of Proposition \ref{prop:boundary epi-perimetric}. Fix $R>0$ and note that due to \eqref{eq:flat norm convergence of blow-up}, passing to a further subsequence if necessary, we may assume that $T_{r_k}-S = R_k + \partial Q_k$ in $B_{R+2}$ and $\sum_{k} \norm{R_k}(B_{R+1}) + \norm{Q_k}(B_{R+1}) < \infty$. Hence there is $R<s<R+1$ such that $\limsup_k \bM(R_k\res B_s) + \bM(\langle Q_k,d,s\rangle) = 0$. Since 
 \[ T_{r_k}\res B_s = S\res B_s + R_k\res B_s - \langle Q_k, d, s \rangle + \partial (Q_k \res B_s) \] 
 we deduce as before by the almost minimality of $T$ in $0$ that 
\[ \norm{T_{r_k}}(B_s) \le (1+ C_0 (sr_k)^{\alpha_0}) \left(\norm{S}(B_s) + \norm{R_k}(B_s) + \mathbf{M}(\langle Q_k, d, s\rangle )\right)\,.\]
Together with the lower semi- continuity we have shown (ii). 
In particular from (ii) and remark \ref{rem:density} we deduce for all $s>0$ that 
\begin{equation}\label{eq:density of the tangent cones} \frac{\norm{S}(B_s)}{\omega_2 s^2} = \lim_{k \to \infty} \frac{\norm{T}(B_{sr_k})}{\omega_2 (sr_k)^2} = \Theta(T,0)=:\Theta_0	
\end{equation}

As a classical consequence of the RHS of \eqref{eq:almost monotonicity formula} applied to $S$ we conclude that $S$ is a cone, compare \cite[Theorem 4.5.3]{Simon}.\\

Let $\mathcal{C}$ be the set of all possible tangent cones of $T$ in $0$. Note that by \eqref{eq:density of the tangent cones} we have $\norm{S}(B_r)= \Theta_0 \omega_2 r^2$ for all $S \in \mathcal{C}$ and $r >0$. For $S\in \mathcal{C}$ let $\mathbf{B}_{\frac{\epsilon_S}{2}}(S)=\{ Z \in \mathbf{I}_2(B_3)\colon \bd_{B_2}(Z,S)< \frac{\epsilon_S}{2}\}$ with $\epsilon_S>0$ given by Proposition \ref{prop:boundary epi-perimetric}. The compactness theorem for integral currents implies that the open cover $\{ \mathbf{B}_{\frac{\epsilon_S}{2}}(S) \colon S \in \mathcal{C} \}$ admits a finite sub-cover
$\{ \mathbf{B}_{\frac{\epsilon_i}{2}}(S_i) \colon i =1, \dots, L\}$. Let $\epsilon:= \min_{i} \frac{\epsilon_i}{2}$. By the arguments above there exists a radius $0< r_* \le \epsilon$ such that for every $r<r_*$ we have $T_r \in \mathbf{B}_{\frac{\epsilon_i}{2}}(S_i)$ for some $i$. Furthermore for a.e. $0<s<r_*$ we have $\langle T_s, d, 1\rangle \in \mathbf{I}_1(\partial B_1)$. But since due to corollary \ref{cor:continuity of the slices in flat norm} the map $s\mapsto \langle T_s, d, 1\rangle$ is continuous we deduce that $\langle T_s, d, 1\rangle\in \mathbf{I}_1(\partial B_1)$ for every $0<s< r_*$. In conclusion $T_r$ satisfies the assumptions of Proposition \ref{prop:boundary epi-perimetric} and we deduce the existence of an $H_r \in \mathbf{I}_2(\partial B_1)$ satisfying \eqref{eq:epiperimetric inequality 2}. Scaling back and multiplying by $r^2$ we obtain
\begin{equation}\label{eq: scaled back epi-perimetric inequality}
\norm{T}(B_r) - \Theta_0 \omega_2 r^2 \le (1-\epsilon_4) \left(\norm{0 \cone \langle T, d, r \rangle }(B_r) - \Theta_0\omega_2 r^2 \right) + C_4 r^{2+\alpha_4}\,. 	
\end{equation}

\emph{Step 2: decay of the spherical excess: } Now we follow \cite[Theorem 3]{White}: Set $f(r):= \norm{T}(B_r) - \Theta_0 r^2$. The function $r \mapsto \norm{T}(B_r)$ is monoton and therefore it has non-negative distributional derivative. Hence $f(r)$ is as well differentiable a.e. with a non-negative singular part of its distributional derivative. Choose $\alpha_5>0$ such that $3\alpha_5 \le \alpha_4$ and $(1+\alpha_5)\le \frac{1}{1-\epsilon_4}$. Since $\norm{0\cone \langle T, d, r \rangle}(B_r) = \frac{r^2}{2}\bM(\langle T, d, r\rangle) \le \frac{d}{dr} \norm{T}(B_r)$ we deduce from \eqref{eq: scaled back epi-perimetric inequality} that 
\[ \frac{d}{dr} \left(\frac{f(r)}{r^{2(1+\alpha_5)}}\right) \ge - \frac{C_5 \omega_2 }{\alpha_5} r^{\alpha_5 -1}\,. \]
Integrating in $r$ and using the short hand notation $e(r):= \frac{f(r)}{\omega_2r^2}$ we get 
\begin{equation}\label{eq: spherical excess deday}
	e(r) \le \left(\frac{r}{s} \right)^{2\alpha_5} e(s) + C_5 s^{3\alpha_5} \text{ for all } 0<r<s< r_*\,.
\end{equation}

\emph{Step 3: decay of the flat norm:} We simply combine \eqref{eq: spherical excess deday} with \eqref{eq:modulus of continuity} to deduce that for $ \frac{s}{2}\le r \le s< r_*$ we have (choosing $\alpha_5$ if necessary smaller)
\begin{equation*}
	\mathcal{F}( \langle T_r, d ,1\rangle - \langle T_s, d, 1\rangle )^2 \le 2\bM(i_\sharp T\res B_s\setminus B_r)^2 +  C s^{2\alpha_1} \le C s^{2\alpha_5}.
\end{equation*}
By iteration on dyadic scales we conclude for all $0< r< s< r_*$  
\begin{equation}\label{eq: flat norm decay I}
	\mathcal{F}( \langle T_r, d ,1\rangle - \langle T_s, d, 1\rangle ) \le 2\bM(i_\sharp T\res B_s\setminus B_r) +  C s^{\alpha_1} \le C s^{\alpha_5} \,.
\end{equation}

\emph{Step 4: Proof of \eqref{eq:flat metric convergence} and \eqref{eq: Haussdorff converegence}:} This step is almost identical with \cite[Theorem 3.1, Step 3]{DeLellis}. But since we need to take care of the boundary part we present the whole argument for the convenience of the reader. 

Let us fix $0<r<s<r_*$ and $\epsilon>0$. Furthermore we define the function $f(t,x):=\frac{\abs{x}}{t}$. Note that $f$ is smooth on the set $t\ge \epsilon$. Consider the current $Q_\epsilon:= \left([\epsilon, 1] \times T\right) \res \{ r \le f < s\}$. By the slicing formula we deduce that 
\begin{align*} \langle Q_\epsilon , f , s\rangle - \langle Q_\epsilon, f , r \rangle &= \partial Q_\epsilon - \left((\delta_1 - \delta_\epsilon)\times T \right) \res \{ r \le f < s\} + ([\epsilon , 1 ]\times \Gamma) \res \{ r \le f < s\}\\
&=\partial Q_\epsilon - R_1 + R_\epsilon + R_\Gamma\,,
\end{align*}
where $R_a= \delta_a \times T\res B_{sa}\setminus B_{ra}$ for $a>0$. We claim that for any $0<\rho<r_*$ we have
\begin{equation}\label{eq:slice of Q_epsilon} \langle Q_\epsilon , f , \rho\rangle = -\left(\frac{\abs{x}}{\rho}, x\right)_\sharp (T\res B_\rho\setminus B_{\epsilon \rho}).\end{equation}
Assuming the claim holds true we can argue as follows: Since for $H(t,x):= t \frac{x}{\abs{x}} = t \,i(x)$ one has $H( \frac{\abs{x}}{\rho}, x) = \frac{1}{\rho} x$, one concludes that
\begin{align}\label{eq: flat norm II}
-T_r\res B_1\setminus B_\epsilon + T_s\res B_1\setminus B_\epsilon &= H_\sharp \langle Q_\epsilon , f , s\rangle - H_\sharp\langle Q_\epsilon, f , r \rangle\\\nonumber
&=\partial H_\sharp Q_\epsilon - H_\sharp R_1 + H_\sharp R_\epsilon + H_\sharp R_\Gamma. 
\end{align}
Let us estimate the mass of the single pieces:
\begin{align*} \bM(H_\sharp Q_\epsilon) &= \int_{\epsilon}^1 \int  \abs{\partial_t H \wedge H_\sharp \vec{T}}  \mathbf{1}_{\{ tr\le \abs{x} < ts\}} \,d\norm{T}\,dt \\
&=\int_{\epsilon}^1 \int_{B_{ts}\setminus B_{tr}} t^2 \abs{i(x)\wedge i_\sharp \vec{T}} \,d\norm{T} \, dt \le \int_{\epsilon}^1 t^2 \,\bM(i_\sharp T\res B_{ts}\setminus B_{tr}) \,dt\\ 
&\le C s^{\alpha_5},
\end{align*}
where we used \eqref{eq: flat norm decay I} in the last step. 
For every $0<a\le 1$ we have again by \eqref{eq: flat norm decay I}
\[\bM(H_\sharp R_a) = \bM (a \,i_\sharp T\res B_{sa}\setminus B_{ra} ) \le a^2 C (as)^{\alpha_5}\,. \]
Finally we estimate $R_\Gamma$ as $H_\sharp Q_\epsilon$ using 
\[ \bM(H_\sharp R_\Gamma) \le \int_{\epsilon}^1 t \,\bM(i_\sharp \Gamma\res B_{ts}\res B_{tr}) \, dt \le C s^{\alpha_1}\,.\]
By combing all the estimates and taking the limit $\epsilon \to 0$, we can conclude that \eqref{eq:flat metric convergence} holds true.
It remains to prove \eqref{eq:slice of Q_epsilon}. It is sufficient to check the identity on differential forms of type $\omega_1=\alpha(t) dt \wedge w_1$ and $\omega_2 = \alpha(t) w_2$ with $\alpha \in C^\infty_c(\R), w_i \in \mathcal{D}^{i}(\R^{n+2}), i=1,2$. We will present the idea for $\omega_1$ the calculation for $\omega_2$ is analogous. For a smooth approximation of the identity $\eta_\delta$ we have
\begin{align*}
	&\langle Q_\epsilon , f , \rho\rangle(\omega_1) = \lim_{\delta \to 0 } \int \langle Q_\epsilon , f , s \rangle(\omega_1) \eta_\delta(s-\rho) \, ds = \lim_{\delta \to 0 } Q_\epsilon( \eta_\delta(f-\rho) f^\sharp ds \wedge \omega_1) \\
	&=\lim_{\delta \to 0} -T\left( \int_\epsilon^1 \eta_\delta\left(\frac{\abs{x}}{t}-\rho\right) \frac{\alpha(t)}{t} \, dt \; d\abs{x}\wedge w_1 \right) = -T\left( \alpha\left(\frac{\abs{x}}{\rho}\right)\frac{\mathbf{1}_{B_\rho \setminus B_{\epsilon \rho}}}{\rho}  d\abs{x}\wedge w_1 \right)\\
	&= - -\left(\frac{\abs{x}}{\rho}, x\right)_\sharp (T\res B_\rho\setminus B_{\epsilon \rho})(\omega_1)\;,
\end{align*}
where we used that 
\[ \lim_{\delta \to 0 }\int_\epsilon^1 \eta_\delta\left(\frac{\abs{x}}{t}-\rho\right) \frac{\alpha(t)}{t} = \alpha\left(\frac{\abs{x}}{\rho}\right)\frac{\mathbf{1}_{B_\rho \setminus B_{\epsilon \rho}}}{\rho}\,. \] 

Finally \eqref{eq: Haussdorff converegence} follows from \eqref{eq:flat metric convergence} by the lower density bound, which is for instance a consequence of the inner monotonicity formula, see \cite[Proposition 2.1]{DeLellis}. Since $\Gamma_{r} \to \Gamma_0$ uniformly in $C^{1,\alpha_1}$ we have only to consider points in $\spt(T_r) \setminus \spt(\Gamma_r)$. Let $y \in \spt(T_r) \setminus \spt(\Gamma_r) \cap B_{\frac12}$ and let $\rho:= \dist(y, \spt(S))$. Note that $\rho<\frac12$. Choose $Q, R$ such that $T_r-S = R + \partial Q$ and $\norm{R}(B_1) + \norm{Q} \le 2 \bd_{B_1}(T_r,S) \le 2C_5 r^{\alpha_5}$. By the slicing theory we may fix $\frac \rho2< \sigma \le \rho$ such that $\bM(\langle Q , d_y,   \sigma \rangle) \le \frac2\rho \norm{Q}(B_\rho(y))$ where $d_y(x):=\abs{x-y}$. Hence by our choice of $\rho$ we have 
\[ T_r\res B_{\sigma}(y) = R\res B_{\sigma}(y) - \langle Q, d_y, \sigma\rangle + \partial ( Q\res B_{\sigma}(y))\,. \]
Now \eqref{eq: Haussdorff converegence} follows from the almost minimality of $T$ and the lower density bound, since  
\[ c \sigma^2 \le \norm{T_r}(B_\sigma(y)) \le \bM(R\res B_\sigma(y) - \langle Q, d_y, \sigma \rangle) \le \frac{2}{\sigma} \bd_{B_1}(T_r,S)\,.\]

\emph{Step 5: Decomposition and proof of (i) - (iii):}
This step follows by the very same argument as in \cite[Theorem 3.1, Step 4]{DeLellis}.

\end{proof}

%\newpage

\section{Lipzschitz approximation}\label{sec.approximation}
In this section we prove a small modification of White's Lipschitz approximation in \cite{White}.
We need to ensure that our Lipschitz approximation preserves the boundary data. We will obtain it as corollary from the fact that we are able to produce Lipschitz approximations that preserve symmetries.

Although our treatment is very close to White's original approach we give a complete account to it.

Suppose $Z$ is a rectifiable curve in $\partial \mathbf{C}_1$, i.e. $Z = X_{\sharp}\a{[0,L]}$. Since $\mathbf{p}\circ X: [0,L] \to S^1$ is continuous there is a lift $\theta: [0,L] \to \R$, unique if we assume that $\theta(0) \in [0, 2\pi[$, such that $\mathbf{p}\circ X(t) = e^{i\theta(t)}$ for all $t$. As a consequence we may consider instead of $X$ the curve $x(t)=(\theta(t), y(t)) \in \R\times \R^{n-2}$. From now on we will denote with small letters, like $x(t)=(\theta(t), y(t))$ the lifted curves and with capital letters like $X(t)$ the curve in $\partial \mathbf{C}_1$.

 We equip $\R\times \R^{n-2}$ with the  euclidean metric $ds^2_e= d\theta^2 + dy^2$ and the metric \begin{equation}\label{eq:metric h} ds^2_h= (1+\abs{y}^2) d\theta^2 + dy^2 + \abs{y\wedge dy}^2= ds_e^2 + ds_{\tilde{h}}^2. \end{equation}
where $ds_{\tilde{h}}^2= \abs{y}^2 d\theta^2 + \abs{y\wedge dy}^2$.  
Hence we have for any $v\in T_{(\theta, y)}\R\times \R^{n-1}$ 
\[ \abs{v}^2_e\le \abs{v}^2_h \le (1+\abs{y}^2) \abs{v}^2_e. \]
The following lemma lists some estimates on geodesics in $\R\times \R^{n-2}$ with respect to the metric $h$. 

\begin{lemma}\label{lem.estimates on geodesics}
	Given two points $x(t_i)=(\theta(t_i), y(t_i)), i=1, 2$ then there is a constant speed geodesic $t\in [t_1, t_2] \mapsto x(t)=(\theta(t), y(t))$ between them with 
	\begin{itemize}
		\item[(i)] $\max_{t \in [t_1, t_2]} \abs{y(t)}= \max_{i=1,2} \abs{y(t_i)}$;
		\item[(ii)] \[\dot{\theta}(t)-\frac{\theta(t_2)-\theta(t_1)}{t_2-t_1} \ge - \frac12 H(x)\max_{i=1,2}{\abs{y(t_i)}}\,. \]
	\end{itemize}  
\end{lemma}
 
\begin{proof}
The geodesic is the unique minimizer of the the energy 
\[ H(z)=\int_{t_1}^{t_2} \abs{\dot{z}}^2_h \]
where the infimum is taken over all curves $t \mapsto z(t)$ with $z(t_i)=x(t_i)$. The Euler Lagrange equation is the classical equation for geodesics i.e.
\begin{equation}\label{eq:geodesic equation} \ddot{x}^i(t) + \Gamma^i(x(t))_{jk}\dot{x}^j\dot{x}^k=0.\end{equation}
The weak formulation is given by 
\[0 = \int_{t_1}^{t_2} h_{ij} \dot{x}^i \dot{v}^j + \partial_lh_{ij}\dot{x}^i\dot{z}^jv \quad \forall v \in C^1_c([t_1, t_2], \R\times \R^n).\]
We claim that $\ddot{\abs{y}^2}\ge 0$. This can be seen using $v=\varphi y$ for a nonnegative  $\varphi \in C^1_c([t_1,t_2])$ in the weak formulation: note that since $y\wedge y =0$ we have $\tilde{h}(\theta,y)(\dot{x}, y) =0$ and since $y \mapsto \tilde{h}$ is two homogeneous we have
\[\partial_lh_{ij}y^l = \partial_l\tilde{h}_{ij}y^l= 2\tilde{h}_{ij}. \]
Combining these we have 
\begin{align*}
	0 &= \int_{t_1}^{t_2} \dot{y}\cdot y \,\dot{\varphi} + \abs{\dot{y}}^2_h \varphi + 2 \tilde{h}_{ij}\dot{x}^i \dot{x}^j \varphi \\
	&\ge\int_{t_1}^{t_2} \dot{y}\cdot y \,\dot{\varphi}.
\end{align*}
In conclusion, we have shown that $t \mapsto \abs{y}^2$ is subharmonic hence (i) follows. 
By direct computations one has $\Gamma^\theta_{\theta \theta} =0$ and $\Gamma^\theta_{\theta i} =\frac{y^i}{1+\abs{y}^2}$. Hence by \eqref{eq:geodesic equation}, Young's inequality and (i) we may estimate
\[ \ddot{\theta} = - \frac{y^i\dot{y}^i\dot{\theta}}{1+\abs{y}^2}\ge- \frac12 \max_{i=1,2}\abs{y(t_i)} \abs{\dot{x}}^2_h \]
Recall that $t\mapsto \abs{\dot{x}}_h^2$ is constant  and hence $H(x)= (t_2-t_1) \abs{\dot{x}}_h^2$. Thus we conclude (ii) from the mean value theorem that 
\begin{align*}
	\dot{\theta}(t)-\frac{\theta(t_2)-\theta(t_1)}{t_2-t_1}=\dot{\theta(t)}-\dot{\theta}(s)= \int_s^t \ddot{\theta} \ge - \frac12 H(x) \max_{i=1,2}\abs{y(t_i)}
\end{align*}
\end{proof}

From now on we will consider closed curves i.e. let $X$ be a closed Lipschitz curve in $\partial \mathbf{C}_1$ of length $L$. After a re-parametrisation, we may assume that $X$ is parametrised by arc-length with respect to the metric $\abs{\cdot}_h$ i.e. if $x(t)$ is the corresponding lift to $\R \times \R^{n-2}$ we have 
\[ \abs{\dot{x}(t)}_h = 1 \quad \forall t. \]
Recall that by the constancy lemma and the fact that $X$ is closed we have that $\mathbf{p}\circ X_\sharp \a{[0,L]} = Q \a{S^1}$, for some $Q \in \N$. In particular this implies for the lift $x(t)$ that
\[ Q= \theta(L)- \theta(0) = \int_0^L \dot{\theta} \]
and $x(t)$ extends to a $L$-periodic function in the sense that
\[ x(t+L)=(\theta(t+L), y(t+L)) = (\theta(t)+ 2\pi Q , y(t)) = x(t) + 2\pi Q e_\theta. \]
For such a curve we have the following Lipschitz approximation, compare \cite[Proposition 2.7]{White}.

\begin{lemma}\label{lem:Lipschitz approximation}
There exist $\epsilon_0>0$ such that if 
\begin{itemize}
\item[(1)] $E:=\int_0^L \abs{\dot{x}}_h - \dot{\theta} \le \epsilon_0$;
\item[(2)] $m:=\sup_{t} \abs{y(t)} \le \epsilon_0$.
\end{itemize}
then for each $C\frac{E}{L} < \delta^2 \le 1-\epsilon_0$ there exists $\tilde{y}:\R \to \R^{n-2}$, $Q$-periodic such that 
if $\tilde{x}(t):= (t, \tilde{y}(t)), \tilde{X}(t):=(e^{it}, \tilde{y}(t))$ one has 
\begin{itemize}
\item[(i)] $\tilde{E}:=\int_{0}^Q \abs{\dot{\tilde{x}}}_h - 1 \le E$;
\item[(ii)] $\sup_{t} \abs{\tilde{y}(t)} \le \sup_{t} \abs{y(t)}=m$;
\item[(iii)] $\operatorname{Lip}(\tilde{y}) \le C\left(\delta +\frac{\sqrt{Em}}{\delta}\right)$;
\item[(iv)] $\mathbf{M}( X_\sharp\a{[a,b]} - \tilde{X}_\sharp\a{[\theta(a),\theta(b)]}) \le C \frac{E}{\delta^2}$ for all $0\le a < b \le L$.
\end{itemize}
If in addition $X$ satisfies an axial symmetry along the $e_1$-axes so does $\tilde{X}$ i.e. if
\begin{equation}\label{eq:axial symmetry}
	 x(-t)= -x(t) \quad \forall t \text{ then } \tilde{x}(-t) = -\tilde{x}(t) \quad \forall t.
\end{equation}
\end{lemma}

\begin{proof}
Consider the non negative $L$-periodic function 
\[ f(t):= \abs{\dot{x}(t)}_h - \dot{\theta}(t) = 1 - \dot{\theta}(t) \]
and the related maximal function 
\[ mf(t):= \sup_{t \in I} \fint_I f. \]
For any $0<\delta<1$ by the classical weak $L^1$-estimate we have for the open set $O:= \{ t \colon mf(t) > \delta^2 \}$ the estimate
\[ \abs{O\cap [0,L]} \le \frac{C}{\delta^2} \int_{-L}^{2L} f \le \frac{3C}{\delta^2} \int_0^L f = C\frac{E}{\delta^2}. \]
Note that by our choice of $\delta$ we have $\abs{O\cap [0,L]} <1$ if $a,b \notin O$ we have 
\begin{equation}\label{eq:Lipschitz estimate part1}
	\fint_a^b f \le \delta^2  \Leftrightarrow \frac{\theta(b)-\theta(a)}{b-a} \ge 1- \delta^2. 
\end{equation}
Hence we conclude that $\abs{X(b)-X(a)}\le \int_a^b \abs{\dot{x}}_e \le \int_a^b \abs{\dot{x}}_h \le (b-a)$ and therefore combined with \eqref{eq:Lipschitz estimate part1}
\begin{equation}\label{eq:Lipschitz estimate part2} \frac{\abs{y(b)-y(a)}}{\abs{\theta(b)-\theta(a)}} \le \frac{\delta}{\sqrt{ 1- \delta^2}}.
\end{equation}
Now let $\tilde{x}$ be the unique minimizer of 
\[ \inf\{ \int_0^L \abs{\dot{z}}_h^2 \colon z(t) = x(t) \forall t \notin O \}. \]
As an open set $O$ is the disjoint union of open intervals i.e.
$O = \bigcup_i ]a_i,b_i[$. Since $f$ is $L$-periodic so is $O$ and since $b_i-a_i < L$ for each $i$ we have that $\tilde{x}$ is as well $L$-periodic. 
Furthermore if $x$ satisfies the axial symmetry we have $f(t)=f(-t)$ and so $O = -O$. But this implies now that the minimizer $\tilde{x}$ will satisfies as well $\tilde{x}(-t)= - \tilde{x}(t)$. 

Note that we have $\tilde{x}(t)= x(t)$ if $t \notin O$ and 
on a interval $]a_i, b_i[$, $\tilde{x}$ is a constant speed geodesic between the points $x(a_i), x(b_i)$. As a consequence of Lemma \ref{lem.estimates on geodesics}, (i), we conclude that $\sup_{t} \abs{\tilde{y}} \le \sup_{t} \abs{y}=m$. Using the minimality of $\tilde{x}$ its consequential constant speed property on intervals $]a_i, b_i[$ we have 
\[ \abs{\dot{\tilde{x}}(t)}_e^2 \le \abs{\dot{\tilde{x}}(t)}_h^2 = \fint_{a_i}^{b_i} \abs{\dot{\tilde{x}}(t)}_h^2 \le \fint_{a_i}^{b_i} \abs{\dot{x}(t)}_h^2 =1 \text{ for } t \in ]a_i,b_i[.\] Furthermore from Lemma \ref{lem.estimates on geodesics}, (ii), we conclude that for each $t \in ]a_i, b_i[$ and \eqref{eq:Lipschitz estimate part1} we have 
\begin{align*} \dot{\tilde{\theta}}(t) &\ge 1-\delta^2 - \frac{1}{2}\sup_{t} \abs{y(t)} \,\int_{a_i}^{b_i} \abs{\dot{\tilde{x}}(t)}_h \\&\ge 1-\delta^2 - \frac{1}{2}m\, (b_i-a_i)\\& \ge 1- \delta^2 - C\frac{mE}{\delta^2}. \end{align*}
We conclude that 
\[ \left(1- \delta^2 - C\frac{mE}{\delta^2}\right)\,\abs{\dot{\tilde{x}}(t)}_e \le \dot{\tilde{\theta}}(t) \text{ for } t \in ]a_i,b_i[. \]
This is equivalent to 
\begin{equation}\label{eq:Lipschitz estimate part3}
\frac{\abs{\dot{\tilde{y}}(t)}^2}{\abs{\dot{\tilde{\theta}}(t)}^2} \le \frac{\delta^2 + C\frac{mE}{\delta^2}}{1- \delta^2 - C\frac{mE}{\delta^2}}.
\end{equation}
Together with \eqref{eq:Lipschitz estimate part2} and $\epsilon_0$ sufficient small, we conclude that $\tilde{x}(t)$ is a Lipschitz graph over $\R$ satisfying (iii) and hence we may find the claimed re-parametrisation $\hat{x}(t) = (t, \hat{y}(t))$.

Finally (iv) follows simply by the estimate on $\abs{O\cap [0,L]}$ since
\begin{align*} &\mathbf{M}( X_\sharp\a{[0,L]} - \tilde{X}_\sharp\a{[0,Q]}) \le \int_{O\cap [0,L]} \abs{\dot{x}}_e + \abs{\dot{\tilde{x}}}_e \le \int_{O\cap [0,L]} \abs{\dot{x}}_h + \abs{\dot{\tilde{x}}}_h\\& \quad\le 2 \int_{O\cap [0,L]} \abs{\dot{x}}_h \le 2 \abs{O\cap [0,L]}\le C \frac{E}{\delta^2}\,.
\end{align*}
\end{proof}

\section{Remarks on semi-calibrated and spherical cross section of an area-minimizing cones}

In \cite{DeLellis} authors showed that, in particular, the following three interesting instances imply the \emph{almost (area) minimality} of the current: area minimizing inside a Riemannian manifold, semicalibrated currents and spherical cross-sections of minimizing currents. 
For the convenience of the reader let us recall their definition:

\begin{definition}\label{def.interesting cases}
	Let $\Sigma \subset \R^{m+n}$ be a $C^2$ submanifold and $U\subset \R^{m+n}$ an open set. An $m$-dimensional integral current $T$ with finite mass and $\spt(T) \subset \Sigma \cap U$ is called % $\Gamma \subset \Sigma$ be a $C^{1,\alpha}$ subman
	\begin{itemize}
	\item[(a)] \emph{area minimizing} in $\Sigma \cap U$ if $\bM(T)\le \bM(T+\partial Q)$ for any $(m+1)$-dimensional current with $\spt(Q)\Subset \Sigma \cap U$.
	\item[(b)] \emph{semicalibrated} in $\Sigma \cap U$ if $\omega_x(\vec{T})=1$ for $\norm{T}$-a.e. $x$ where $\omega$ is a $C^1$ $m$-dimensional semicalibration (in $\Sigma$) i.e. $\norm{\omega_x}_c\le 1$ at every $x \in \Sigma$, where $\norm{\cdot}_c$ denotes the comass norm on $\Lambda^mT_x\Sigma$.
	\item[(c)] a \emph{spherical cross section} of an area-minimizing cone if $\Sigma = \partial B_R$, $U= \R^{m+n}$, and $0\cone T$ is area minimizing.
	\end{itemize}
The almost minimality in the sense of Definition \ref{def.almost minimzing} now follows from the following proposition. The reader should compare with the interior situation covered in \cite{DeLellis} and note that only minor modifications are needed.
\begin{proposition}\label{prop.almost minimality in the interesting cases}
Under the assumptions of Definition \ref{def.interesting cases} and $\partial T\res U = \Gamma \res U$, with $\Gamma \subset \Sigma$ a $m-1$-dimensional $C^{1,\alpha}$-submanifold, we have that $T$ is almost minimimzing in the sense of Definition \ref{def.almost minimzing}. Additionally when cases (b) or (c) apply, then, for any $Q$ supported in $B_r(x)$
\begin{equation}\label{eq.case b +c}
\norm{T}(B_r(x)) \le \norm{T+\partial Q}(B_r(x)) + C \norm{Q}(B_r(x)).
\end{equation}
Furthermore the first variation formulas for any vector field $X\in C^1_c(U, \R^{m+n}$ tangent to $\Gamma$ has the following expressions, respectively,
\begin{itemize}
\item[(a)] $\delta T(X)= - \int X\cdot \vec{H}_T(x) \, d\norm{T}$, where $\vec{H}_T$ can be calculated using the second fundamental form $A_\Sigma$ of $\Sigma$ by $\vec{H}_T(x)=\sum_{i=1}^m A_\Sigma(v_i,v_i)$, with $v_i$ orthonormal basis of $\vec{T}(x)$; 	
\item[(b)] $\delta T(X) = T ((d\omega)\res X)$;
\item[(c)] $\delta T(X)= \int X\cdot m \frac{x}{R^2} \, d\norm{T}$.
\end{itemize}
\end{proposition}
\begin{proof}
Let us start with the almost minimality:\\
\emph{in case (a):} it is the calculation presented in \cite[Proof of Proposition 0.4, case (a)]{DeLellis}	 and just using $\operatorname{Lip}(\mathbf{p}) \le 1 + C \norm{A_\Sigma}_\infty r $\\
\emph{in case (b) + (c):} In \cite[Proposition 1.2]{DeLellis} it had been shown that \eqref{eq.case b +c} holds. In its proof it was not used that $\partial T\res U= 0$. Now it remains to show that it implies almost minimality. We may assume that $\bM(\partial Q)< 2 \norm{T}(B_r(x))$ otherwise \eqref{eq:almost minimizing} holds trivially. Furthermore replacing $Q$ by $x\cone \partial Q$ we may  assume that $\bM(Q)\le \frac{r}{m+1}\bM(\partial Q)$. Hence we conclude $\bM(Q)\le \frac{r}{m+1}\bM(\partial Q) \le \frac{2}{m+1} r \norm{T}(B_r(x))$. So we can reabsorb it on the left hand side in \eqref{eq.case b +c} to deduce that
\[(1-Cr) \norm{T}(B_r(x)) \le \norm{T+\partial Q}(B_r(x))\,. \]
It remains to conclude  the first variation formulas: case (a) is classical, see for instance \cite{Simon}. case (c) is a special case of (a) since $A_{\partial B_R}(v,w) = - \frac{x}{R^2} \, v\cdot w$. Finally we give a short proof for the case $(b)$. Let $\varphi_t$ be the flow generated by $X$ hence we conclude using the fact that $T$ is semicalibrated by $\omega$ that for all $t$
\begin{align*}
\norm{(\varphi_t)_\sharp T}(B_r(x)) - \norm{T}(B_r(x)) \ge T\res{B_r(x)}(\varphi_t^\sharp \omega - \omega)\,.
\end{align*}
The left hand side is $t \,\delta T(X) + o(t)$. By the definition of the   Lie-derivative, have $\varphi_t^\sharp \omega - \omega = t \,\mathcal{L}_X\omega + o(t)$. Hence we conclude that $t \,\delta T(X)\ge t\, T\res{B_r(x)}(\mathcal{L}_X \omega)-o(t)$. This holds for all small $t$ so $\delta T(X)= T\res{B_r(x)}(\mathcal{L}_X \omega)$. Since $\mathcal{L}_X\omega = (d\omega)\res X - d(\omega\res X)$ we have 
\[ T\res{B_r(x)}(\mathcal{L}_X \omega) = T((d\omega)\res X) - \Gamma(\omega \res X)\,.\]
Note that since $X$ is tangent to $\Gamma$ we have that $\Gamma(\omega \res X)=0$. This concludes the proof. 
\end{proof}

\end{definition}

\bibliography{boundarycones}

\begin{thebibliography}{1}

\bibitem{AllardAlmgren}
W.~K. Allard and F.~J. Almgren, Jr.
\newblock The structure of stationary one dimensional varifolds with positive
  density.
\newblock {\em Invent. Math.}, 34(2):83--97, 1976.

\bibitem{DeLellisDePhilippis}
Camillo De~Lellis, Guido De~Philippis, Jonas Hirsch, and Annalisa Massaccesi.
\newblock Boundary regularity of mass-minimizing integral currents and a
  question of {A}lmgren.
\newblock In {\em 2017 {MATRIX} annals}, volume~2 of {\em MATRIX Book Ser.},
  pages 193--205. Springer, Cham, 2019.

\bibitem{DeLellis}
Camillo De~Lellis, Emanuele Spadaro, and Luca Spolaor.
\newblock Uniqueness of tangent cones for two-dimensional almost-minimizing
  currents.
\newblock {\em Comm. Pure Appl. Math.}, 70(7):1402--1421, 2017.

\bibitem{Spolaor}
Max Engelstein, Luca Spolaor, and Bozhidar Velichkov.
\newblock ({L}og-)epiperimetric inequality and regularity over smooth cones for
  almost area-minimizing currents.
\newblock {\em Geom. Topol.}, 23(1):513--540, 2019.

\bibitem{Federer}
Herbert Federer.
\newblock {\em Geometric measure theory}.
\newblock Die Grundlehren der mathematischen Wissenschaften, Band 153.
  Springer-Verlag New York Inc., New York, 1969.

\bibitem{Reifenberg}
E.~R. Reifenberg.
\newblock An epiperimetric inequality related to the analyticity of minimal
  surfaces.
\newblock {\em Ann. of Math. (2)}, 80:1--14, 1964.

\bibitem{Simon}
Leon Simon.
\newblock {\em Lectures on geometric measure theory}, volume~3 of {\em
  Proceedings of the Centre for Mathematical Analysis, Australian National
  University}.
\newblock Australian National University, Centre for Mathematical Analysis,
  Canberra, 1983.

\bibitem{White}
Brian White.
\newblock Tangent cones to two-dimensional area-minimizing integral currents
  are unique.
\newblock {\em Duke Math. J.}, 50(1):143--160, 1983.

\end{thebibliography}
\bibliographystyle{plain}
\end{document}